\documentclass[10pt,a4paper,english,DIV=9]{scrartcl}
\pdfoutput=1

\usepackage[T1]{fontenc}
\usepackage[utf8]{inputenc}
\usepackage{babel}

\usepackage[sc]{mathpazo}   
\usepackage[scaled=.95]{helvet}
\usepackage{courier}

\usepackage{microtype}

\KOMAoptions{DIV=last}						

\usepackage{amssymb,amsthm, amsmath, mathtools}			

\usepackage{dsfont}						

\usepackage{tikz}							
\usetikzlibrary[matrix,arrows]

\usepackage{hyperref}
\hypersetup{
  colorlinks   = true,          
  urlcolor     = blue,    
  linkcolor    = blue,       
  citecolor   = purple        
}

\let\OLDthebibliography\thebibliography
\renewcommand\thebibliography[1]{
  \OLDthebibliography{#1}
  \small
  \setlength{\parskip}{0pt}
  \setlength{\itemsep}{1.5pt plus 0.3ex}
}

 \theoremstyle{plain}
 \newtheorem{thm}{Theorem}[section]
  \theoremstyle{definition}
  \newtheorem{defn}[thm]{Definition}
  \theoremstyle{definition}
  
  \theoremstyle{plain}
  \newtheorem{lem}[thm]{Lemma}
  \theoremstyle{plain}
  \newtheorem{cor}[thm]{Corollary}
  \theoremstyle{remark}
  \newtheorem{rem}[thm]{Remark}
  \theoremstyle{plain}
  \newtheorem{prop}[thm]{Proposition}
  \theoremstyle{plain}
  
  \theoremstyle{remark}
  
  \theoremstyle{remark}
  
  \theoremstyle{plain}
  
  \theoremstyle{definition}

\DeclareMathOperator{\PP}{PP}
\DeclareMathOperator{\Hom}{Hom}

\DeclareMathOperator{\Spec}{Spec}
\DeclareMathOperator{\Trop}{Trop}
\DeclareMathOperator{\trop}{trop}
\DeclareMathOperator{\Star}{Star}
\DeclareMathOperator{\cone}{cone}

\DeclareMathOperator{\relint}{relint}
\DeclareMathOperator{\ini}{in}

\DeclareMathOperator{\codim}{codim}
\DeclareMathOperator{\ord}{ord}
\DeclareMathOperator{\divv}{div}

\DeclareMathOperator{\Lin}{Lin}

\DeclareMathOperator{\id}{id}

\DeclareMathOperator{\B}{B}

\newcommand{\mB}{\mathcal B} 
\newcommand{\nr}{N_{\mathds R}}

\renewcommand{\phi}{\varphi}

\let\oldoverline\overline
\renewcommand\overline[1]{\oldoverline{\mkern-1mu #1}}

\addtokomafont{section}{\large\rmfamily}
\addtokomafont{subsection}{\normalsize\rmfamily}
\addtokomafont{title}{\rmfamily\mdseries\scshape}	
\addtokomafont{paragraph}{\rmfamily}

\title{\large Intersection Theory on Linear Subvarieties of Toric Varieties}

\author{\normalsize Andreas Gross}
\date{}

\begin{document}

\maketitle
\vspace{-1.2cm}

\begin{abstract}
We give a complete description of the cohomology ring $A^*(\overline Z)$ of a compactification of a linear subvariety $Z$ of a torus in a smooth toric variety whose fan $\Sigma$ is supported on the tropicalization of $Z$. It turns out that cocycles on $\overline Z$ canonically correspond to Minkowski weights on $\Sigma$ and that the cup product is described by the intersection product on the tropical matroid variety $\Trop(Z)$.
\end{abstract}

\section{Introduction}

In \cite{FS97}, Fulton and Sturmfels show that the cohomology ring of a complete toric variety $X_\Sigma$ is isomorphic to the ring of Minkowski weights on the corresponding fan $\Sigma$. The product of two Minkowski weights was described by the so-called ``fan displacement rule''. Later it was shown in \cite{Katz12} and \cite{Rau08} that the product described by this rule is equal to the intersection product in tropical geometry introduced in \cite{AR10}. Here, we prove an analogous result about closures of linear subvarieties of tori in toric varieties. Let $T$ be a torus over an algebraically closed field $K$ with trivial valuation, and let $M$ be its character lattice. By linear subvarieties of $T$ we mean those varieties which are cut out by an ideal $I\unlhd K[M]$ which is linear in $K[\mathds Z^n]$ in the canonical sense after choosing an appropriate isomorphism $K[M]\cong K[\mathds Z^n]$. We compactify $Z$ in a smooth toric variety $X=X_\Sigma$ whose fan has support $|\Sigma|=|\Trop(Z)|$. This restriction on $\Sigma$ is made to ensure that $\overline Z$ intersects the torus orbits of $X$ properly. Our main theorem states that in this situation  the cohomology ring $A^*(\overline Z)$ can be described by the group $M^*(\Sigma)$ of Minkowski weights on $\Sigma$. Furthermore, the multiplication is induced by the intersection product on the tropical matroid variety $\Trop(Z)$, which has been introduced in \cite{FR10} and \cite{Sh13}. 

\begin{thm}
\label{iThm:I_Z is ring isomorphism}
Let $Z$ be a linear subvariety of the torus $T$ and let $\Sigma$ be a unimodular fan with $|\Sigma|=|\Trop(Z)|$. 
Then there is a canonical ring isomorphism $\mathcal I_Z:A^*(\overline Z)\rightarrow M^*(\Sigma)$, where the ring structure on $M^*(\Sigma)$ is induced by the tropical intersection product on $\Trop(Z)$.
\end{thm}

Compactifications of linear subvarieties of tori have usually been studied in the context of hyperplane arrangements. After an appropriate choice of coordinates any linear subvariety $Z$
of a torus $T$ is equal to the intersection of a linear subspace $V$ of some projective space $\mathds P^n$
that is not contained in any coordinate hyperplane with the big open torus $\mathds G^{n+1}_m/\mathds G_m$. The
choice of a linear parametrization $\mathds P^r \cong V$ identifies $Z$ with the complement of the hyperplane arrangement in $\mathds P^r$ consisting of the preimages of the coordinate hyperplanes of $\mathds P^n$.
Constructing well-behaved compactifications of complements of hyperplane arrangements
is a standard problem of algebraic geometry. A very well-known solution has been given by
De Concini and Procesi in \cite{DP95}, who associate \emph{wonderful} compactifications $\overline Z_{\mathcal G}$ to so-called
building sets $\mathcal G$ in the intersection lattice of the arrangement. They also succeeded in giving a
combinatorial description of the cohomology of these compactifications, which has been further simplified in \cite{FY04}. Another approach to compactify Z was introduced by Tevelev in \cite{Tev07}.
For connected hyperplane arrangements he introduced the tropical compactification $\overline Z _{\trop}$,
which has further been studied in \cite{HKT06}. Their results proved that complements of connected
hyperplane arrangements are what is called \emph{hübsch} in \cite{HKT09}, a very desirable property that
implies, among other things, that the closure $\overline Z$ in any toric variety whose fan is supported
on $|\Trop(Z)|$ intersects the torus orbits smoothly.

This paper is organized as follows. We begin Section \ref{tSec} by reviewing the basic constructions
of the tropical theory. Afterwards, we look at Minkowski weights in greater detail. The main
result of this section is of technical nature, but will be very useful in the proof of Theorem
\ref{iThm:I_Z is ring isomorphism}. For its statement let $\Sigma$ be a complete unimodular fan in $\mathds R^n$ which has a subfan $\widetilde\Sigma \subseteq \Sigma$
that defines a matroid variety $A$ after assigning weight $1$ to all its maximal cones. Writing
$Z^*(\mathds R^n)$ and $Z^*(A)$ for the graded rings of tropical fan cycles in $\mathds R^n$ and $A$, respectively,
we get a map $Z^*(\mathds R^n) \to Z^*(A)$ which assigns to $C \in Z^*(\mathds R^n)$ the intersection product
$C \cdot A$ in $\mathds R^n$. Using piecewise polynomials, it easily follows from the results of \cite{F11,KP08} that
this map is a ring epimorphism. But as we are interested in Minkowski weights, the question
arises if this still holds if we replace $Z^*(\mathds R^n)$ by $M^*(\Sigma)$ and $Z^*(A)$ by $M^*(\widetilde \Sigma)$. Theorem \ref{tThm:Intersection with Tropicalization is surjective} provides an affirmative answer to this. As a direct consequence we obtain a new proof of the
fact that $M^*(\widetilde\Sigma)$ is a subring of $Z^*(A)$. To our knowledge, this result has first been proven in
\cite{Sh13} using tropical modifications, whereas our proof uses the interplay between intersection
theory on toric varieties and tropical intersection theory.

In Section \ref{mSec} we start with the examination of compactifications of linear subvarieties of a
torus $T$. Given a linear subvariety $Z \subseteq T$, we will only consider compactifications in smooth
toric varieties such that $\overline Z$ is proper and intersects all torus orbits properly. As shown in \cite{Gub12},
this happens if and only if the fan $\Sigma$ corresponding to the toric variety has a subfan  $\widetilde\Sigma$ which
is supported on $\Trop(Z)$. Note that in contrast to the statement of Theorem \ref{iThm:I_Z is ring isomorphism}, this condition
allows fans with support strictly larger than that of $\Trop(Z)$. This gives us more flexibility
in the choice of the ambient toric variety: in particular, we can choose it to be complete.
However, the closure of $Z$	 does not meet the torus orbits corresponding to cones in $\Sigma$ that
are not contained in $\widetilde\Sigma$, so $\overline Z$ only depends on $\widetilde \Sigma$. Hence allowing larger fans does not make
the statement more general.

Now given $\Sigma$, we will show that the orbit structure $X_\Sigma = \bigcup_{\sigma\in\Sigma} O(\sigma)$ induces a stratification $\overline Z = \bigcup_{\sigma\in\Sigma} \overline Z \cap O(\sigma)$ and that each stratum $\overline Z \cap O(\sigma )$ is a linear subvariety of
the torus $O(\sigma)$. We use this stratification to define the morphism $\mathcal I_Z\colon A^k(\overline Z) \to M^k(\Sigma)$: a
cocycle $c \in A^k(\overline Z)$ is mapped to the Minkowski weight $\sigma \mapsto \deg(c \cap [\overline Z \cap V (\sigma )])$, where
$V(\sigma)$ denotes the closure of the torus orbit $O(\sigma)$. The morphism $\mathcal I_Z$ is analogous to the one
constructed by Fulton and Sturmfels, and, in fact, we recover their isomorphism if we set
$Z = T$. We will prove injectivity and surjectivity of $\mathcal I_Z$ separately. For injectivity the crucial
point is that the Kronecker duality homomorphism $\mathcal D_Z\colon A^k(\overline Z) \to \Hom(A^k(\overline Z), \mathds Z)$, which
assigns to $c \in A^k(\overline Z)$ the morphism $\alpha \mapsto \deg(c \cap \alpha)$, is an isomorphism. This is wrong for
general varieties, but Totaro succeeded to prove that it is in fact true for what he calls linear
schemes \cite{Tot95}. We will use Totaro's result by proving that $\overline Z$ is linear in his sense. To prove the
surjectivity of $\mathcal I_Z$, we embed $X_\Sigma$ in a complete toric variety $X_\Delta$ and show that the pull-back
map $A^*(X_\Delta) \to A^*(\overline Z)$ corresponds to intersection with the tropical cycle $\Trop(Z)$ on the
level of Minkowski weights. Then the statement follows from the fact that the resulting map
$M^*(\Delta) \to M^*(\Sigma)$ is onto, which was our main result of Section \ref{tSec}.

I would like to thank Andreas Gathmann for many helpful discussions and comments.

\section{Preliminaries}
\label{tSec}

In this section we want to recall the basic definitions and properties of the objects and
operations in tropical geometry needed in this paper. We also prove a lifting theorem for
Minkowski weights which is essential for our analysis of the intersection theory of linear
subvarieties of toric varieties in the next section, but may also be of independent interest.

\subsection{Tropical intersection theory}

The tropical objects occurring in this paper are tropical (fan) cycles in the sense of \cite{AR10} on the one hand and Minkowski weights on the other. Both objects will always live in the real vector space $\nr=N\otimes_{\mathds Z}\mathds R$ for a lattice $N$. \emph{Tropical cycles} are equivalence classes of certain \emph{weighted fans}, that is pure-dimensional fans with weights on their top-dimensional cones, where two weighted fans are called equivalent if and only if they have a common refinement which respects the weights.  Note that when speaking of a ``cone'' we actually mean a rational polyhedral cone. The equivalence classes we consider are exactly those of \emph{tropical fans}, which are those weighted fans satisfying the so-called \emph{balancing condition}. To explain the balancing condition we first need the concept of lattice normal vectors. Suppose that $\tau$ is a codimension one face of a cone $\sigma$. If we write $N_\delta=N\cap \Lin(\delta)$ for a cone $\delta\subseteq\nr$, then $N_\sigma/N_\tau$ is a one-dimensional lattice. The \emph{lattice normal vector} $u_{\sigma,\tau}$ of $\sigma$ with respect to $\tau$ is the unique generator of $N_\sigma/N_\tau$ which is contained in the image of $\sigma$ in $(N_\sigma/N_\tau)_{\mathds R}=\Lin(\sigma)/\Lin(\tau)$. Now if $\mathcal A=(\Sigma,\omega)$ is a weighted fan with underlying fan $\Sigma$ and weight function $\omega$, and for $k\in\mathds N$ we denote by $\Sigma^{(k)}$ the codimension $k$ cones of $\Sigma$ [that is the cones of dimension $\dim(\Sigma)-k$], then $\mathcal A$ satisfies the balancing condition if and only if we have
\begin{equation*}
\sum_{\mathclap{\sigma:\tau\leq\sigma\in\Sigma^{(0)}}}~~\omega(\sigma)u_{\sigma,\tau} =0 \text{ in }\nr/\Lin(\tau)
\end{equation*}
for every $\tau\in\Sigma^{(1)}$. The set of all tropical cycles of given codimension $k$ in $\nr$ can be given the structure of an Abelian group $Z^k(\nr)$ \cite[Lemma 2.14]{AR10}.

If $\Sigma$ is a fan, weighted fan or tropical cycle, then we can define its support $|\Sigma|$. For a fan this is just the union of all its cones. For a weighted fan it is the union of all maximal cones with nonzero weight. Finally, if $\Sigma$ is a tropical cycle, then its support is defined to be the support of any representative (this is easily seen to be well-defined). In any of these cases we can define the subgroup $Z^k(\Sigma)$ of $Z^{k+\codim(\Sigma,\nr)}(\nr)$ consisting only of those cycles whose support is contained in that of $\Sigma$.

An important operation on fans is that of taking stars. For any cone $\sigma\subseteq\nr$ define $N(\sigma)=N/N_\sigma$. If $\sigma$ is a cone of a fan $\Sigma$, we define the \emph{star}  of $\Sigma$ at $\sigma$ to be the fan $\Star_\Sigma(\sigma)$ in $N(\sigma)_{\mathds R}$ having as cones the images of the cones of $\Sigma$ containing $\sigma$ under the projection map $\nr\rightarrow N(\sigma)_{\mathds R}$. In case there are weights defined on $\Sigma$ we equip $\Star_\Sigma(\sigma)$ with the induced weights. As the projection respects lattice normal vectors, the star will be tropical if $\Sigma$ is.

The class of tropical cycles of greatest interest in this paper is that of matroid varieties. For every loopfree matroid $\mathcal M$ on $E=\{1,\dots, n\}$ there is a pure-dimensional fan $\mB(\mathcal M)$ in $\mathds R^n=\mathds Z^n\otimes_{\mathds Z}\mathds R$. Its cones are of the form 
\begin{equation*}
\langle \mathcal F\rangle =\cone\{v_{F_1},\dots, v_{F_{k-1}}, v_E, -v_E\},
\end{equation*}
where $\mathcal F=\{\emptyset \subsetneq F_1 \subsetneq\dots\subsetneq F_k=E\}$ is a chain of flats, $v_F=-\sum_{f\in F}e_f$, and $e_1,\dots, e_n$ denotes the standard basis of $\mathds R^n$. Assigning weight $1$ to all maximal cones of $\mB(\mathcal M)$ we obtain a tropical fan. These fans always have lineality space $\Lin\{v_E\}$. As we also want to consider fans without lineality space we define a \emph{matroidal fan} to be any tropical fan isomorphic to $\mathcal B(\mathcal M)/\Lin\{v_E\}$ for some loopfree matroid $\mathcal M$. Here, by an isomorphism of two fans $\mathcal A\subseteq\nr$ and $\mathcal B\subseteq N'_{\mathds R}$ we mean a linear map $\nr\rightarrow N'_{\mathds R}$ which maps $|\mathcal A|$ bijectively onto $|\mathcal B|$ in a way compatible with the weights and which is induced by a lattice morphism $N\rightarrow N'$ that maps the lattice $N_{\mathcal A}=\Lin|\mathcal A|\cap N$ isomorphically onto $N'_{\mathcal B}=\Lin|\mathcal B|\cap N'$. \emph{Tropical matroid varieties} are then defined to be exactly those tropical cycles which are associated to matroidal fans.

Taking for the matroid $\mathcal M$ the uniform matroid of rank $n$ on $n$ elements yields a complete fan $\mB(\mathcal M)$. In particular, $\mathds R^n$, or more generally $\nr$ for any lattice $N$, equipped  with the trivial weight $1$, is a matroid variety.

Matroid varieties play an important role in tropical intersection theory because they allow
intersection products. That is, for each matroid variety $\B$, the graded Abelian group $Z^*(\B)=\bigoplus_{k\in\mathds N}Z^k(\B)$ can be given the structure of a graded commutative ring with unity $\B$, which also has the property that the support of a product $C\cdot D$ is contained in $|C|\cap |D|$ \cite[Thm. 4.5]{FR10}, \cite[Prop. 3.13]{Sh13}.

Intersection products on matroid varieties also respect intersections with piecewise polynomials, which are the tropical analogues of equivariant cocycles on toric varieties (see Section \ref{subsec:Minkowski weights}). A \emph{piecewise polynomial} of degree $k$ on a cycle $A$ is a continuous function $\phi:|A|\rightarrow\mathds R$ for which there exists a fan structure $\Sigma$ on $|A|$ such that $\phi$ is given by a homogeneous polynomial of degree $k$ with integer coefficients on every cone of $\Sigma$. The Abelian group of piecewise polynomials of degree $k$ on $A$ is denoted by $\PP^k(A)$. Allowing also polynomials with mixed degrees we obtain a graded ring $\PP^*(A)=\bigoplus_{k\in\mathds N}\PP^k(A)$. In \cite[Prop. 2.24]{F11} it is shown that $Z^*(A)$ can be given the structure of a graded $\PP^*(A)$-module, making $Z^*(A)$ a $\PP^*(A)$-algebra in case $A$ is a matroid variety. Given two cycles $A$, $B\in Z^*(\nr)$ with $|A|\subseteq |B|$, there is a natural restriction morphism $\PP^*(B)\rightarrow \PP^*(A)$. This induces a $\PP^*(B)$-module structure on $Z^*(A)$ and with this structure, the inclusion $Z^*(A)\rightarrow Z^*(B)$ [note that this is a graded morphism of degree $\codim(A,B)$] is a morphism of $\PP^*(B)$-modules.

\subsection{Minkowski weights}
\label{subsec:Minkowski weights}

Minkowski weights have been introduced by Fulton and Sturmfels in \cite{FS97} to describe the
Chow cohomology ring of a complete toric variety. They differ from cycles in that they do not allow refinements. To be more precise, given a pure-dimensional fan $\Sigma$, the set of codimension $k$ Minkowski weights $M^k(\Sigma)$ is defined as the set of weight functions $c:\Sigma^{(k)}\rightarrow\mathds Z$ making the codimension $k$ skeleton of $\Sigma$ a tropical fan. This is easily seen to be a subgroup of the group of maps $\Sigma^{(k)}\rightarrow \mathds Z$.
We will give a short presentation of the connection of these
weights to the intersection theory of toric varieties. First, we assume that $\Sigma$ is complete, in which case the associated toric variety $X = X(\Sigma)$ is complete as well. Then for every $k\in\mathds N$ there is the Kronecker duality homomorphism
\begin{equation*}
\mathcal D\colon A^k(X) \to \Hom(A_k(X), \mathds Z),\;\; c \mapsto (\alpha \mapsto \deg(c \cap \alpha)),
\end{equation*}
which is an isomorphism by \cite[Thm. 3]{FMSS95}, \cite[Thm. 2]{Tot95}. By the proposition on page 96 in \cite{F89} the Chow group $A_k(X)$ is generated by the cycle classes $[V(\sigma)]$ of the closures of the orbits corresponding to the codimension $k$ cones $\sigma \in \Sigma^{(k)}$. This implies that a morphism $c \in \Hom(A_k(X), \mathds Z)$ is uniquely determined the images $c([V(\sigma)])$. Using this we obtain a monomorphism
\begin{equation*}
\mathcal I\colon A^k(X) \to \mathds Z^{\Sigma^{(k)}},\;\; c \mapsto (\deg(c \cap [V(\sigma)]))_{\sigma\in\Sigma^{(k)}}
\end{equation*}
from $A^k(X)$ into the group of maps $\Sigma^{(k)} \to \mathds Z$. A description of the image of $\mathcal I$ amounts to determining the relations of the cycle classes $[V(\sigma)]$ in $A_k(X)$. These are given by the divisors of the characters of the torus orbits $O(\sigma)$ for $\sigma\in\Sigma$. Namely, the relations in $A_k(X)$ are given by
\begin{equation*}
0 = \divv(\chi^m) = \sum_{\sigma:\tau\leq\sigma\in\Sigma^{(k)}}\langle m, u_{\sigma,\tau}\rangle [V(\sigma)]\quad\text{in }A_k(X),
\end{equation*}
where $\tau\in\Sigma^{(k+1)}$ and $m \in M(\tau) = M \cap \tau^\perp$ \cite[Prop. 1.1]{FS97}. These relations are dual to the balancing condition. Therefore, $\mathcal I$ induces an isomorphism of $A^*(X)$ and $M^*(\Sigma)$.
The cup-product on $A^*(X)$ induces a multiplication on $M^*(\Sigma)$ which can be described
combinatorially by the ``fan displacement rule'' \cite[Prop.3.1, Thm.3.2]{FS97}.Because Minkowski
weights satisfy the balancing condition by definition, there is a canonical inclusion map
$M^*(\Sigma) \hookrightarrow Z^*(\Sigma), c \mapsto [c]$ for arbitrary fans $\Sigma$. In \cite[Thm. 4.4]{Katz12}  and \cite[Thm. 1.9]{Rau08} it has
been shown that if $\Sigma$ is complete, this map is in fact a ring homomorphism. Thus $M^*(\Sigma)$ can be considered as a subring of $Z^*(\nr)$ in this case, giving an algebraic interpretation of tropical intersection theory.

A connection like this exists for intersections with piecewise polynomials as well. For
a complete fan $\Sigma$ denote by $\PP^*(\Sigma)$ the subring of $\PP^*(\nr)$ consisting of all piecewise polynomials which restrict to polynomials with integer coefficients on the cones of $\Sigma$. It has been shown in \cite[Thm.\ 5.4]{Br97} for $\Sigma$ unimodular and in \cite[Thm. 1]{P06} for general $\Sigma$ (both
results do not need $\Sigma$ to be complete) that there is a canonical isomorphism $\PP^*(\Sigma) \to A^*_T(X)$ between piecewise polynomials and equivariant cocycles on $X$. The composite $\eta^*\colon \PP^*(\Sigma) \to A^*(X)$ of this isomorphism with the canonical morphism $A^*_T(X) \to A^*(X)$ has been described combinatorially in \cite{KP08}. The intersection product of tropical cycles with piecewise polynomials introduced in \cite{F11} has been modeled to be compatible with $\eta^*$. More
precisely, if $A \in Z^*(\nr)$ is represented by a Minkowski weight $c \in M^*(\Sigma) = A^*(X)$, and $\phi \in \PP^*(\Sigma)$, then
\begin{equation*}
\phi\cdot A=[\eta^*\phi\cup c].
\end{equation*}

\subsection{Initial ideals and tropicalization}
\label{subsec:initial ideals}

Let $f \in K[M]$ and $u \in \nr$. The \emph{initial form} $\ini_u(f)$ of $f$ with respect to $u$ is the sum of all terms of $f$ corresponding to monomials with minimal $u$-weight. That is, if we write
$f = \sum_i a_i \chi^{m_i}$ as a finite sum with distinct $m_i$ and all $a_i$ nonzero, then $\ini_u(f)$ is the sum of all terms $a_i \chi^{m_i}$ for which $\langle m_i, u\rangle$ is minimal. Let $I \unlhd K[M]$ be an ideal. Then the \emph{initial ideal} $\ini_u(I)$ of $I$ is defined by $\ini_u(I) = \langle \ini_u(f) \mid f \in I\rangle$.

Now let $Z$ be a pure-dimensional subscheme of the algebraic torus $T = \Spec K[M]$ with
corresponding ideal $I \unlhd K[M]$. Following \cite{TropBook} we define the (set-theoretic) \emph{tropicalization} of $Z$ by 
\begin{equation*}
\trop(Z) = \{u \in \nr \mid \ini_u(I) \neq K[M]\}.
\end{equation*}
The tropicalization $\trop(Z)$ can be written as the support of a purely $\dim(Z)$-dimensional fan
$\Sigma$ \cite[Thm.\ 3.3.8]{TropBook}. This fan can be chosen in such a way that for all maximal cones $\sigma\in\Sigma$ the initial ideals $\ini_u(I)$ are independent of the choice of $u\in\relint(\sigma)$ \cite[Prop. 3.2.8]{TropBook}. So we get a well-defined initial ideal $\ini_\sigma(I)$. Assigning to a maximal cone $\sigma \in\Sigma$ the sum of the multiplicities of the irreducible components of $\Spec(K[M]/ \ini_\sigma(I))$ we obtain a balanced
weighted fan \cite[Thm.\ 3.4.14]{TropBook}. It clearly depends on the choice of $\Sigma$, but its tropical cycle is well-defined and we call it the tropicalization $\Trop(Z)$ of $Z$.

Tropicalization of subvarieties of tori over fields with trivial valuation is closely related to
the intersection theory of toric varieties: Given the subvariety $Z$ of $T$, there exists a unimodular fan $\Sigma$ such that the closure $\overline Z$ in the toric variety $X = X(\Sigma)$ intersects all torus orbits in the expected dimension \cite[Thm. 1.2]{Tev07}. By Poincaré duality \cite[Cor.\ 17.4]{F98} there exits a unique cocycle $c \in A^*(X)$ such that $c \cap [X] = [\overline Z]$. This cycle corresponds to a Minkowski weight on $\Sigma$ whose corresponding tropical cycle is equal to $\Trop(Z)$ \cite[Lemma 3.2]{ST08}, \cite[Lemma 2.3]{KP11}. 

In the case when $T = \mathds G^n_m \subseteq \mathds A^n$ and $Z$ is given by a set of linear equations, that is $Z = \{x \in T \mid Ax + b = 0\}$ for some $m \times n$ matrix $A$ and some vector $b \in K^m$, the tropicalization of $Z$ can be described combinatorially: Let $B$ be a matrix whose rows span the kernel of $(A | b)$. Then the tropicalization of $Z$ is given by the matroid variety associated to the matroid represented by $B$ \cite[Section 4.1]{TropBook}.

\subsection{Lifting Minkowski weights}

Now we want to prove an adaption to the world of Minkowski weights of the following lifting
property of cycles in a matroid variety to a larger matroid variety.

\begin{lem}
\label{tLem:Intersection is Ring homomorphism}
Let $A$, $B\in Z^*(\nr)$ be matroid varieties with $|A|\subseteq |B|$, then the map
\begin{align*}
\phi\colon Z^*(B)\rightarrow Z^*(A),\;\;C\mapsto C\cdot A,
\end{align*}
which maps a cycle $C$ to its intersection product in $B$ with $A$, is a surjective $\PP^*(B)$-algebra homomorphism.
\end{lem}

\begin{proof}
The fact that $\phi$ is a morphism of $\PP^*(B)$-modules follows directly from the compatibility properties stated above.  To prove that $\phi$ also respects the ring structures we first note that the structure maps $\PP^*(A)\rightarrow Z^*(A)$ and $\PP^*(B)\rightarrow Z^*(B)$ are surjective by \cite[Remark 3.2]{F11}.  Also, as a direct consequence of \cite[Prop. 2.8]{F11}, the restriction morphism $\PP^*(B)\rightarrow\PP^*(A)$ is surjective. This shows that both, $Z^*(A)$ and $Z^*(B)$ are homomorphic images of $\PP^*(B)$. Together with the fact that the unity $B$ of $Z^*(B)$ is mapped to $A$, the unity of $Z^*(A)$, this proves the lemma.
\end{proof}

To make an analogous statement about Minkowski weights, we have to fix a fan structure
on $B$ that respects $A$. More precisely, we treat unimodular fans $\Sigma$ with support equal to $|B|$ having a subfan $\widetilde\Sigma$ with support equal to $|A|$ and replace $Z^*(B)$ by $M^*(\Sigma)$ and $Z^*(A)$ by $M^*(\widetilde\Sigma)$ in the lemma.  In the case of Minkowski weights it is advantageous to reduce to the case where $\Sigma$ is complete, because then we can use methods from algebraic intersection theory. That this reduction is possible follows from the fact that given three tropical cycles $A, B, C$ with $|A|\subseteq |B|\subseteq |C|$, such that $B$ and $C$ are matroid varieties  with intersection products $\cdot_B$ and $\cdot_C$, respectively, we have
\begin{equation*}
(D\cdot_C B)\cdot_B A=D\cdot_C A
\end{equation*}
for all cycles $D\in Z^*(C)$.

\begin{thm}
\label{tThm:Intersection with Tropicalization is surjective}
Let $A\in Z^*(\nr)$ be a matroid variety and let $\Sigma$ be a complete unimodular fan having a subfan $\widetilde\Sigma\subseteq\Sigma$ with $|\widetilde\Sigma|=|A|$. Then the map
\begin{align*}
M^*(\Sigma)\rightarrow M^*(\widetilde\Sigma),\;\; c\mapsto c\cdot A
\end{align*}
is surjective.
\end{thm}

\begin{proof}
After dividing by the lineality space of $\Sigma$ we can assume that $\Sigma$ consists of strongly convex cones. In this way, we ensure that the toric variety $X_\Sigma$ associated to $\Sigma$ has the correct dimension. Let $c\in M^k(\widetilde\Sigma)$. By Lemma \ref{tLem:Intersection is Ring homomorphism} there is a cycle $B\in Z^k(\nr)$ such that $[c]=B\cdot A$ (remember that $[c]$ denotes the tropical cycle associated to $c$). There is a complete fan $\Delta$ refining $\Sigma$, such that $B=[b]$ for a Minkowski weight $b\in M^k(\Delta)$. Let $\pi:X_\Delta\rightarrow X_\Sigma$ be the toric morphism induced by the identity on $N$, and let $s$ be the Minkowski weight on $\Sigma$ with $[s]=A$. Identifying Minkowski weights and cocycles on complete toric varieties, $\pi$ induces a morphism $\pi^*:M^k(\Sigma)\rightarrow M^k(\Delta)$. By applying the projection formula, one easily sees that $\pi^*$ is nothing but the refinement of Minkowski weights. Then the equation $[b]\cdot A=[c]$ translates into 
\begin{equation}
\label{tForm:before pushforward}
b\cup\pi^*(s)=\pi^*(c),
\end{equation}
where we write the intersection product as ``cup''-product to emphasize that we think of Minkowski weights as cocycles on toric varieties. Since $X_\Sigma$ is smooth and $\pi$ is proper there is the Gysin push-forward $\pi_*:M^k(\Delta)\rightarrow M^k(\Sigma)$ \cite[pp. 328-329]{F98}. Applying it to Equation \ref{tForm:before pushforward} we obtain
 \begin{equation*}
 \pi_*(b)\cup s= \pi_*(b\cup \pi^*(s))=\pi_*\pi^*(c),
 \end{equation*}
where the first equality follows from \cite[($G_4$) (i), p. 26]{FM81}. By \cite[($G_3$) (ii), p. 26]{FM81} (with $g=\id$ and $\theta=[\pi]$) we have $\pi_*\pi^*(c)=\pi_*[\pi]\cup c$. So if we show that $\pi_*[\pi]=1$ we get $\pi_*(b)\cdot A=\pi_*(b)\cup s =c$ and are done. We have
\begin{equation*}
\pi_*[\pi]\cap [X_\Sigma]=\pi_*([\pi]\cap [X_\Sigma])
\end{equation*}
by the definition of the push-forward of bivariant classes. By \cite[Ex. 17.4.3 (c)] {F98} and \cite[Cor. 8.1.3]{F98} we have $[\pi]\cap[X_\Sigma]=[X_\Delta]$. Hence 
\begin{equation*}
\pi_*[\pi]\cap [X_\Sigma]=\pi_*[X_\Delta]=[X_\Sigma]
\end{equation*}
and we are done by Poincaré duality \cite[Cor. 17.4]{F98}.
\end{proof}

Now we can easily show that the tropical intersection product on a matroid variety induces a ring structure on the group of Minkowski weights on any unimodular fan with the same support. This statement also follows from \cite[Prop. 3.13]{Sh13}, where it has been proved using the purely tropical method of tropical modifications.

\begin{cor}
\label{tCor:Ring structure on Minkowski weights}
Let $\Delta$ be a pure-dimensional unimodular fan in $\nr$ which defines a matroid variety $A$ if we assign weight $1$ to all of its maximal cones. Then the group of Minkowski weights $M^*(\Delta)$ can be given a ring structure which is compatible with the intersection product on $Z^*(A)$.
\end{cor}

\begin{proof}
By \cite[Thm. 2.8, p. 75]{E96} and \cite[Thm. 11.1.9]{CLS} there is a complete unimodular fan $\Sigma$ with $\Delta\subseteq \Sigma$. Now consider the commutative diagram
\begin{center}
\begin{tikzpicture}[auto]
\matrix[matrix of math nodes,row sep=1.1cm]{
|(MS)|M^*(\Sigma) &[1.2cm] |(Znr)| Z^*(\nr) \\
|(MD)|M^*(\Delta) & |(ZA)| Z^*(A). \\
};
\begin{scope}[->,font=\footnotesize]
\draw (MS)--node{$\_\cdot A$} (MD);
\draw (Znr)--node{$\_\cdot A$} (ZA);
\end{scope}

\begin{scope}[right hook->]
\draw (MS) -- (Znr);
\draw (MD) -- (ZA);
\end{scope}
\end{tikzpicture}
\end{center}
All groups except $M^*(\Delta)$ occurring in this diagram have a ring structure and the morphisms at the top and on the right are ring homomorphisms. The one on the right even is surjective by Lemma \ref{tLem:Intersection is Ring homomorphism} and so is the map on the left by Theorem \ref{tThm:Intersection with Tropicalization is surjective}. Together with the fact that the horizontal maps are injective it follows that $M^*(\Delta)$ is a subring of $Z^*(A)$.
\end{proof}

\section{Linear Subvarieties of Tori}
\label{mSec}

Let us first make precise what we mean when speaking of a linear subvariety $Z$ of a torus $T$. If $T$ is equal to $\Spec K[\mathds Z^n]=\Spec K[x_1^{\pm1},\dots, x_n^{\pm1}]$ for some $n$, it would be natural to call a polynomial $f\in K[M]$ linear if it is of the form $f=a_0+\sum_i a_i x_i$ for some coefficients $a_i\in K$. As the character lattice $M$ of $T$ is always isomorphic to $\mathds Z^n$ for some $n$, the assumption $T=\Spec K[\mathds Z^n]$ is not really a restriction, and we could use this for our definition of linear by calling a subvariety $Z\subseteq T$ linear if its vanishing ideal is generated by linear polynomials. However, if we do so there will arise some difficulties: If $\Sigma$ is a fan in $\nr$, where $N$ is the lattice dual to $M$, and we take the closure $\overline Z$ of a linear subvariety $Z$ in the toric variety $X_\Sigma$, then we would like say (and in fact will prove in Proposition \ref{mProp:linear stratification}) that the subvariety $\overline Z\cap O(\sigma)$ of the torus $O(\sigma)$ is linear again. But the coordinates on $\mathds Z^n$ do not induce natural coordinates on the character lattice $M(\sigma)=M\cap \sigma^\perp$ of $O(\sigma)$ so that it is unclear what linear should mean in this context. This problem does not occur if we instead go with the following definition.
\begin{defn}
\label{mDef:Linear Ideal}
Let $M$ be a lattice and $B=\{e_1, \dots, e_n\}$ a basis for $M$. We call $f\in K[M]$ \emph{linear
with respect to $B$} if it is of the form 
\begin{align*}
	f=a_0+\sum\limits_{i=1}^n a_i\chi^{e_i}.
\end{align*}
for some $a_i\in K$, where $\chi^m$ denotes the basis vector of $K[M]$ corresponding to $m\in M$. An ideal $I\unlhd K[M]$ is said to be \emph{linear with respect to $B$} if it is generated by polynomials which are linear with respect to $B$. We will call $I$ \emph{linear} if it is linear with respect to some basis. A subvariety $Z\subseteq T$ will be called linear (w.r.t.\ some basis $B$), if its vanishing ideal in $K[M]$ is.
\end{defn}

We wish to describe the Chow cohomology $A^*(\overline Z)$ of the closure of a linear variety $Z\subset T$ in a toric variety $X=X_\Sigma$ as Minkowski weights on the fan $\Sigma$. In analogy to the result of Fulton and Sturmfels, the isomorphism between Chow cocycles and Minkowski weights should use the stratification of $X_\Sigma$ into torus orbits  and assign to a cocycle $c \in A^k(\overline Z)$ the Minkowski weight with multiplicity $\deg(c\cap [ \overline{\overline Z\cap O(\sigma)}])$ on a cone $\sigma\in\Sigma$ with $\dim(\overline Z\cap O(\sigma))=k$. But this construction cannot make sense for arbitrary $\Sigma$. The two things that can go wrong are first, that $\overline Z$ could be non-complete so that the degree in the formula above is not defined. By \cite[Prop 2.3]{Tev07}, this happens if and only if the tropicalization of $Z$ is not contained in the support of $\Sigma$. The second problem is that Minkowski weights on $\Sigma$ only assign weights to cones of a fixed dimension. So the cones $\sigma$ with $ \dim(\overline{\overline Z\cap O(\sigma)})=k$ should all have the same dimension and, furthermore, $\overline Z\cap O(\sigma)$ should be pure-dimensional. We see that our approach only makes sense if $\overline Z$ intersects all torus orbits properly. It is a direct consequence of the results proved in \cite[Chapter 14]{Gub12} that this is the case if and only if for each $\sigma\in\Sigma$ we either have $\relint(\sigma)\cap|\Trop(Z)|=\emptyset$ or $\sigma\subseteq|\Trop(Z)|$. These considerations lead to the following definition.

\begin{defn}
\label{mDef:admissible fan}
Let $Z$ be a subvariety of the torus $T$. A fan $\Sigma$ in $N_\mathds R$ is called \emph{admissible} (for $Z$) if it is unimodular, and $|\Trop(Z)|$ is a union of cones of $\Sigma$. If $\Sigma$ is admissible for $Z$, we denote by $\widetilde\Sigma=\{\sigma\in\Sigma\mid\sigma\subseteq|\Trop(Z)|\}$ the subfan of $\Sigma$ consisting of all cones contained in the tropicalization of $Z$. Note that the cones of $\widetilde\Sigma$ are exactly those cones of $\Sigma$ with $\overline Z\cap O(\sigma)\neq\emptyset$ \cite[Lemma 2.2]{Tev07}.
\end{defn}

\subsection{The stratification of $\overline Z$}

As we have just seen, admissible fans are exactly those which are unimodular and for which the closure $\overline Z$ in the corresponding toric variety is complete and intersects all torus orbits properly. The following result is an easy consequence of this.

\begin{lem}
\label{mLem:Intersection and Closure Commute}
Let $Z\subseteq T$ be a subvariety of the torus, $\Sigma$ an admissible fan for $Z$ and $\sigma\in\Sigma$ a cone in $\Sigma$. Then we have the set-theoretic equality
\begin{align*}
\overline {\overline Z\cap O(\sigma)} =\overline Z\cap V(\sigma)
\end{align*}
in the toric variety $X=X_\Sigma$, where $V(\sigma)$ denotes the closure of the torus orbit $O(\sigma)$ corresponding to $\sigma$.
\end{lem}

\begin{proof}
Let $Y$ be an irreducible component of $\overline Z\cap V(\sigma)$ and let $\tau\in\Sigma$ be a cone which is maximal with the property that $\sigma\leq\tau$ and $Y\subseteq V(\tau)$. Then $Y$ has nonempty intersection with the orbit $O(\tau)$. Because $O(\tau)$ is open in $V(\tau)$, the set $Y\cap O(\tau)$ is open in $Y$, showing that $Y$ has dimension at most $\dim(Z)-\dim(\tau)$. On the other hand, because $X$ is smooth, every component of $\overline Z\cap V(\sigma)$ has dimension at least $\dim(Z)-\dim(\sigma)$. We conclude that $\sigma=\tau$, showing that every component of $\overline Z\cap V(\sigma)$ intersects $O(\sigma)$. The equality we want to prove follows.
\end{proof}

We recall that a stratification of a variety $X$ is a finite decomposition $X= \coprod S_i$ such that each stratum $S_i$ is locally closed, and the ``boundary'' $\overline{S_i}\setminus S_i$ is a union of strata of lower dimensions. The following corollary will show that every admissible compactification of a subvariety $Z\subseteq T$ has a natural stratification. Afterwards, we will show that in case $Z$ is linear, the strata of this stratification are, in fact, linear subspaces of tori.

\begin{cor}
\label{mCor:Stratification passes to Z}
Let $\Sigma$ be an admissible fan for a subvariety $Z\subseteq T$. Then the closure of $Z$ in $X_\Sigma$ is stratified by the subsets $\{\overline Z\cap O(\sigma)\mid\sigma\in\widetilde\Sigma\}$. For $0\leq k\leq n$, the strata of dimension $k$ in this stratification are exactly those with $\sigma\in\widetilde\Sigma^{(k)}$.
\end{cor}

\begin{proof}
We  have seen that the closures of $Z$ in toric varieties corresponding to admissible fans intersect all orbits properly. Together with the fact that $\dim(\Trop(Z))=\dim(Z)$, this proves the statement about dimensions. The rest follows from Lemma \ref{mLem:Intersection and Closure Commute}.
\end{proof}

\begin{prop}
\label{mProp:linear stratification}
Let $Z\subseteq T$ be linear and $\Sigma$ an admissible fan. Then for each cone $\sigma\in\widetilde\Sigma$ the intersection $\overline Z\cap O(\sigma)$ is a  linear subvariety of the torus $O(\sigma)$.
\end{prop}

\begin{proof}
We prove the statement by induction on the dimension of $\sigma$. If $\dim(\sigma)=0$, then $\overline Z\cap O(\sigma)=Z$ is linear by assumption. Hence we can suppose $\dim(\sigma)>0$. In this case we find a face $\tau\in\widetilde\Sigma$ of $\sigma$ of dimension $\dim(\tau)=\dim(\sigma)-1$. By induction we know that $\overline Z\cap O(\tau)$ is linear, and by Lemma \ref{mLem:Intersection and Closure Commute} we have $\overline Z\cap O(\sigma)=\overline{\overline Z\cap O(\tau)}\cap O(\sigma)$. By \cite[Prop. 14.3]{Gub12}, the support of the tropicalization of $\overline Z\cap O(\tau)$ is equal to that of $\Star_{\widetilde\Sigma}(\tau)$, showing that $\Star_{\Sigma}(\tau)$, the fan corresponding to the toric variety $V(\tau)$, is admissible for $\overline Z\cap O(\tau)$. We see that by considering the toric variety $V(\tau)$ we can assume that $\dim(\sigma)=1$.

The torus orbit $O(\sigma)$ is contained in the affine toric variety $U_\sigma=\Spec K[\sigma^\vee\cap M]$. If $Z$ is given by the ideal $I\unlhd K[M]$, the intersection $I\cap K[\sigma^\vee\cap M]$ is the vanishing ideal of the closure of $Z$ in $U_\sigma$. The orbit $O(\sigma)$ is embedded into $U_\sigma$ via the morphism
\begin{align*}
\phi\colon K[\sigma^\vee\cap M]\rightarrow K[M(\sigma)],\;\; \chi^u\mapsto\begin{cases}
\chi^u  &\text{, if $u\in\sigma^\bot$}\\
0       &\text{, else}.
\end{cases}
\end{align*}
It follows that the ideal of $\overline Z\cap O(\sigma)$ in the coordinate ring of $O(\sigma)$ is given by the radical of $I_\sigma\coloneqq\phi(I\cap K[\sigma^\vee\cap M])$. We claim that $I_\sigma$ is equal to $\ini_u(I)\cap K[M(\sigma)]$. Whenever the image $\phi(f)$ of an $f\in K[\sigma^\vee\cap M]$ is nonzero, the $u$-weight of $f$ must be zero. Hence $\ini_u(f)$ is equal to the sum of terms of $f$ with $u$-weight zero, which are exactly those which do not go to zero when applying $\phi$. It follows that $\phi(f)=\ini_u(f)\in\ini_u(I)\cap K[M(\sigma)]$. Conversely, if $0\neq g\in\ini_u(I)\cap K[M(\sigma)]$, then $g$ is $u$-homogeneous and therefore there exists some $f\in I$ with $g=\ini_u(f)$. With the same argument as before we get $f\in K[\sigma^\vee\cap M]$ and $g=\ini_u(f)=\phi(f)\in I_\sigma$.

Let $B$ be a basis of $M$ with respect to which $I$ is linear. It is a consequence of \cite[Prop. 1.4.4]{CA07} that the initial ideal $\ini_u(I)$ is again linear with respect to $B$, say generated by the linear polynomials $f_1,\dots, f_k$. Because $\ini_u(I)$ is $u$-homogeneous, we can assume all of the $f_i$ to be homogeneous with respect to $u$. Let $\mathcal M$ be the set of all $m\in M$ such that the monomial $\chi^m$ occurs in one of the $f_i$. We define an equivalence relation on $\mathcal M$ by letting $m\sim m'$ if and only if $\langle m, u\rangle=\langle m',u\rangle$, and for each equivalence class $x\in \mathcal M/\sim$ we choose a representative $r(x)\in x$. As each $f_i$ is $u$-homogeneous, all of the exponents occurring in $f_i$ are equivalent, hence there is a well defined equivalence class $[f_i]\in \mathcal M/\sim$ with generator $r(i)=r([f_i])$. This choice of $r(i)$ ensures that $\chi^{-r(i)}f_i\in K[M(\sigma)]$. Hence, $\ini_u(I)$ is generated in $K[M(\sigma)]$, namely by the polynomials $\chi^{-r(1)}f_1,\dots, \chi^{-r(k)}f_k$, and since $K[M]$ is free as $K[M(\sigma)]$-module, the ideal $\ini_u(I)\cap K[M(\sigma)]$ is generated by these polynomials as well. Let $B'=\{a-r(x)\mid x\in\mathcal M/\sim,~ a\in x\setminus\{r(x)\}\}$. This is certainly a subset of $M(\sigma)$, and if we add the elements $\{r(x)\mid x\in\mathcal M/\sim,~ 0\notin x\}$ we obtain a basis of the sublattice of $M$ generated by the elements in $\mathcal M$. Since this sublattice is saturated, we can also complete $B'$ to a basis $\widetilde B$ of $M(\sigma)$. By construction, the generators for $\ini_u(I)\cap K[M(\sigma)]$ from above are all linear with respect to this basis, showing that $\overline Z\cap O(\sigma)$ is linear.
\end{proof}

Before we start to consider the Chow groups and the Chow cohomology of linear subvarieties of tori, let us state the following immediate consequence of the preceding proposition, which will be of great importance in the proof of Theorem \ref{mThm:Pullback is Tropicalization}.

\begin{cor}
\label{mCor:Tropiclization and is local}
Let $Z$ be a linear subvariety of the torus $T$, and $\mathcal A=(\Sigma,\omega)$ a unimodular tropical fan representing the cycle $\Trop(Z)$. Furthermore, let $\sigma\in\Sigma$. Then the tropicalization of the subvariety $\overline Z\cap O(\sigma)$ of the torus $O(\sigma)$ is represented by $\Star_{\mathcal A}(\sigma)$.
\end{cor}

\begin{proof}
By \cite[Prop. 14.3]{Gub12}, the underlying sets of $\Trop(\overline Z\cap O(\sigma))$ and $\Star_\mathcal A(\sigma)$ are equal, so we just need to show that the weights coincide. But $Z$ is linear by assumption and $\overline Z\cap O(\sigma)$ by Proposition \ref{mProp:linear stratification}, hence all weights are $1$ in both cases.
\end{proof}

\begin{rem}
The preceding lemma also follows from known results about extended tropicalization and tropical toric varieties: Following Kajiwara \cite{Kaj08} and Payne \cite{P09b} one can construct
a tropical toric variety $\trop(X_\Sigma)$ associated to $\Sigma$ and a natural tropicalization map
\begin{equation*}
\trop\colon X \to \trop(X).
\end{equation*}
The stratification of $X$ into orbits induces a stratification $\trop(X) = \coprod \trop(O(\sigma))$ and each stratum $\trop(O(\sigma))$ is canonically isomorphic to $N(\sigma)_{\mathds R}$. Furthermore, the restriction of $\trop$ to the torus $O(\sigma)$ is equal to the ordinary tropicalization map. This works for any valued field, but to give the desired result we should work over a field extension of $K$ with surjective valuation. Then the image $\trop(\overline Z)$ is the closure of $\trop(Z)$ in $\trop(X)$ by \cite[Lemma 3.1.1]{OP13}. Hence
\begin{equation*}
\trop(\overline Z \cap O(\sigma)) = \overline{\trop(Z)} \cap \trop(O(\sigma)),
\end{equation*}
which is equal to $\Star_{\mathcal A}(\sigma)$ by \cite[Lemma 3.1.2]{OP13}.
\end{rem}

\subsection{Intersection theory on $\overline Z$}

Our first result on the intersection theory of linear subvarieties of toric varieties will give us generators for the Chow groups.

\begin{cor}
\label{mCor:Chow group generators}
Let $Z\subseteq T$ be linear, $\Sigma$ an admissible fan for $Z$, and $0\leq k\leq \dim(Z)$. Then the $k$-th Chow group of the closure of $Z$ in $X_\Sigma$ is generated by the cycles $[\overline Z\cap V(\sigma)]$ for $\sigma\in\widetilde\Sigma^{(k)}$.
\end{cor}

\begin{proof}
By Proposition \ref{mProp:linear stratification} we know that all strata of the canonical stratification of $\overline Z$ introduced in Corollary \ref{mCor:Stratification passes to Z} are linear subvarieties of some torus. In particular, they are all quasi-affine,
that is isomorphic to an open subset of some affine space. The statement now follows from
the well-known result that the Chow groups of varieties with a quasi-affine stratification are
generated by the closures of the strata	\cite[Prop. 1.17]{EH16}.
\end{proof}

Now we know generators of the Chow groups of $\overline Z$, but we do not know any relations between them. The next lemma will be the essential ingredient to change this.

\begin{lem}
\label{mLem:ChowGroupRelations}
Let $I\unlhd K[M]$ be a linear ideal and $Z$ the corresponding subvariety of the torus. Furthermore, let $\Sigma$ be an admissible fan for $Z$, let $\sigma\in\widetilde\Sigma$ be a ray in $\Sigma$ with primitive generator $u_\sigma$, and let $m\in M$. Then $\chi^m\in K[M]$ defines a rational function on the closure of $Z$ in $X_\Sigma$ and
\begin{align*}
\ord_{\overline{Z}\cap V(\sigma)}(\chi^m)=\langle m, u_\sigma\rangle.
\end{align*}
In particular, we have $\divv(\chi^m)=\sum_\sigma\langle m,u_\sigma\rangle [\overline Z\cap V(\sigma)]$, where the sum is taken over all rays of $\widetilde\Sigma$.

\end{lem}

\begin{proof}
The character $\chi^m$ clearly defines a rational function on $\overline Z$. To prove the formula we can assume that $X_\Sigma=U_\sigma$ in which case we have $\overline Z\cap V(\sigma)=\overline Z\cap O(\sigma)$. Let $\phi:K[\sigma^\vee\cap M]\rightarrow K[M(\sigma)]$ be the morphism corresponding to the closed embedding $O(\sigma)\rightarrow U_\sigma$. In the course of the proof of Proposition \ref{mProp:linear stratification} we showed that the image $I_\sigma=\phi(I\cap K[\sigma^\vee\cap M])$ is prime in $K[M(\sigma)]$ (it is linear with respect to some basis). Therefore, the preimage of $I_\sigma$ in $K[\sigma^\vee\cap M]$, which is equal to $\ker{\phi}+I\cap K[\sigma^\vee\cap M]$, is prime, too. Choose a basis $e_1^*,e_2^*,\dots,e_n^*$ of $N$ with dual basis $e_1,\dots,e_n\in M$ such that $e_1^*=u_\sigma$. With these coordinates, $K[\sigma^\vee\cap M]=K[\chi^{e_1},\chi^{\pm e_2},\dots, \chi^{\pm e_n}]$ and $\ker(\phi)=\langle \chi^{e_1} \rangle$. Consequently, the ideal in the coordinate ring $K[\sigma^\vee\cap M]/I\cap K[\sigma^\vee\cap M]$ of $\overline Z $ cutting out $\overline Z\cap O(\sigma)$ is generated by $\chi^{e_1}$. Hence, $\mathcal{O}_{\overline{Z}\cap O(\sigma),\overline{Z}}$ is a discrete valuation ring whose maximal ideal is generated by $\chi^{e_1}$. Writing
\begin{align*}
m=\langle m, u_\sigma\rangle e_1+\sum\limits_{i=2}^n\langle m, e_i^*\rangle e_i
\end{align*}
and noting that $\chi^{e_i}$ is a unit in $\mathcal{O}_{\overline{Z}\cap O(\sigma),\overline{Z}}$ for $i\geq 2$
we obtain $\ord_ {\overline{Z}\cap V(\sigma)}(\chi^m)=\langle m, u_\sigma\rangle$.

Because $\chi^m$ is an invertible regular function on $Z$, the divisor $\divv(\chi^m)$ is a linear combination of prime divisors contained in $\overline Z\setminus Z$. But this is exactly the union of the varieties $\overline Z\cap V(\sigma)$ for rays $\sigma\in\widetilde\Sigma$, which yields the ``in particula'' statement.
\end{proof}

\begin{rem}
Corollary \ref{mCor:Chow group generators} and Lemma \ref{mLem:ChowGroupRelations} show us that there is a strong analogy between the Chow groups of toric varieties and those of admissible compactifications of linear subvarieties of tori. The Chow groups of both toric varieties and their linear subvarieties have generators corresponding to cones of a fan. Furthermore, in both cases we get a relation between those generators for all cones $\tau$ in the respective fan and $m \in M(\tau)$: In the toric case we take the divisor of the rational function $\chi^m$ on the toric variety $V(\tau)$, whereas in the case of linear varieties we take the divisor of the rational function $\chi^m$ on the linear variety
$\overline Z \cap V (\tau)$. A comparison between the formula given in Section \ref{subsec:Minkowski weights} for the divisor in the toric case, and the formula from the preceding lemma for the linear case shows that they are completely analogous. The main difference between the two cases is that the torus action on the toric variety assures that all relations between the generators are of the given form \cite[Thm.\ 1]{FMSS95}, whereas this is no longer clear for linear subvarieties of toric varieties. However, by explicitly describing the duals of the Chow groups by Minkowski weights we will see in Corollary \ref{mCor:Chow groups of Z modulo torsion} that up to potential torsion this is indeed true for linear varieties as well.
\end{rem}

To describe the Chow cohomology of an admissible compactification $\overline Z$ in a toric variety $X_\Sigma$ we proceed similarly to Fulton and Sturmfels in \cite{FS97}. Given a cocycle $c\in A^k(\overline Z)$, we will first apply the Kronecker duality homomorphism $\mathcal D_Z:A^k(\overline Z)\rightarrow \Hom(A_k(\overline Z),\mathds Z)$ which assigns to $c$ the morphism mapping $\alpha\in A_k(\overline Z)$ to $\deg(c\cap\alpha)$. The Chow group $A_k(\overline Z)$ is generated by the cycles $[\overline Z\cap V(\sigma)]$ for $\sigma\in\widetilde\Sigma^{(k)}$, hence $\mathcal D_Z(c)$ is uniquely determined by its images on those cycles. Let us denote the induced map $\widetilde\Sigma^{(k)}\rightarrow \mathds Z$ by $\mathcal I_{Z}(c)$. For every $\tau\in\widetilde\Sigma^{(k+1)}$, the subvariety $\overline Z\cap V(\tau)$ is an admissible compactification of the linear variety $\overline Z\cap O(\tau)$ by Lemma \ref{mLem:Intersection and Closure Commute}, Proposition \ref{mProp:linear stratification},
 and Corollary \ref{mCor:Tropiclization and is local}. Consequently, we know that
 \begin{equation*}
 \sum_{\mathclap{\sigma:\tau\leq\sigma\in\widetilde\Sigma^{(k)}}}~~\langle m,u_{\sigma,\tau}\rangle [\overline Z\cap V(\sigma)] =0 \quad\text{in } A_k(\overline Z)
 \end{equation*}
for all $m\in M(\tau)$ by Lemma \ref{mLem:ChowGroupRelations} (remember that $u_{\sigma,\tau}$ denotes the lattice normal vector of $\sigma$ with respect to $\tau$). As noted earlier in Section \ref{subsec:Minkowski weights}, these relations are dual to the balancing
condition.  We conclude that $\mathcal I_{Z}(c)$ actually is a Minkowski weight on $\widetilde\Sigma$, that is an element in $M^k(\widetilde\Sigma)$. This shows that our construction really yields a morphism of Abelian groups
\begin{align*}
\mathcal I_{Z}:A^*(\overline Z)\rightarrow M^*(\widetilde\Sigma)
\end{align*}
analogous to the one constructed by Fulton and Sturmfels for complete toric varieties, and, in fact, we recover the isomorphism of \cite{FS97} if we take $Z=T$. But we do not know yet that the image $\mathcal I_Z(c)$ of a cocycle $c\in A^*(\overline Z)$ completely determines $c$. This will be ensured by the next proposition.

\begin{prop}
\label{mProp:Kronecker map is Isomorphism}
Let $\Sigma$ be admissible for a linear subvariety $Z\subseteq T$. Then the Kronecker duality map $\mathcal D_Z$ is an isomorphism. In particular, $\mathcal I_Z$ is injective.
\end{prop}

\begin{proof}
The statement follows from Theorem 2 of Burt Totaro's paper \cite{Tot95}, which says that the
Kronecker duality map is an isomorphism for all complete varieties (or rather schemes) of a
certain type. The type of schemes considered by Totaro is that of what he calls linear schemes.
The class of schemes which are linear in his sense is defined recursively. We do not want to
go into the details here. In our setting it is sufficient to note that schemes stratified by linear
schemes are again linear and that the complement of a union of affine subspaces in an ambient
affine space is linear \cite[p. 5]{Tot95}. By Corollary \ref{mCor:Stratification passes to Z}, $\overline Z$ is stratified by $\{\overline Z \cap O(\sigma) \mid \sigma\in\widetilde\Sigma\}$. Each of the strata $\overline Z \cap O(\sigma)$ is cut out by equations which are linear in some coordinates and is hence isomorphic to the intersection of an affine subspace $L$ of some $\mathds A^k$ with the $k$-dimensional torus $T' \subseteq \mathds A^k$. This can also be considered as the complement of a finite union of affine subspaces of $L$ and therefore it is linear in the sense of Totaro. Consequently, $\overline Z$, too, is linear in Totaro's sense. Applying Totaro's theorem we see that $\mathcal D_Z$ is an isomorphism.
The ``in particular'' statement then follows directly from the construction of $\mathcal I_Z$.
\end{proof}

\begin{cor}
\label{mCor:A^0(overline Z)=mathds Z}
Let $\Sigma$ be admissible for a linear subvariety $Z \subseteq T$. Then $A^0(\overline Z)\cong \mathds Z$.
\end{cor}

\begin{proof}
Consider the composite morphism $A^0(\overline Z)\xrightarrow{\mathcal I_Z}M^0(\widetilde\Sigma)\hookrightarrow Z^0(\widetilde\Sigma)$, which is one-to-one by what we just saw. We have $|\widetilde\Sigma|=|\Trop(Z)|$ and the cycle $\Trop(Z)$ is, as it is a matroid variety, irreducible by \cite[Lemma 2.4]{FR10}. Therefore, $Z^0(\widetilde\Sigma)$ is freely generated by $\Trop(Z)$. The tropical cycle $\Trop(Z)$ is represented by the Minkowski weight in $M^0(\widetilde\Sigma)$ having weight $1$ on all maximal cones. Noting that this is exactly the image of $1\in A^0(\overline Z)$ under $\mathcal I_Z$ finishes the proof.
\end{proof}

Now that we know that $\mathcal I_Z$ embeds the Chow cohomology of $\overline Z$ into the group of Minkowski weights on $\widetilde\Sigma$, our next goal is to show that $\mathcal I_Z$ even is an isomorphism. One of the key ingredients for this is the following theorem.

\begin{thm}
\label{mThm:Pullback is Tropicalization}
Let $\Sigma$ be a complete fan which is admissible for a linear subvariety $Z\subset T$, and let $X=X_\Sigma$ be its associated complete toric variety. Then the diagram
\begin{center}
\begin{tikzpicture}[auto]
\matrix[matrix of math nodes,row sep=1.1cm]{
|(AX)|A^*(X) &[1.2cm] |(MS)| M^*(\Sigma) \\
|(AZ)|A^*(\overline Z) & |(MwS)| M^*(\widetilde\Sigma), \\
};
\begin{scope}[->,font=\footnotesize]
\draw (AX)--node{$\mathcal I_T$} (MS);
\draw (AZ)--node{$\mathcal I_{Z}$} (MwS);
\draw (AX)--node{$i^*$} (AZ);
\draw (MS)--node{$\_\cdot\Trop(Z)$} (MwS);
\end{scope}
\end{tikzpicture}
\end{center}
where the vertical map on the right assigns to a Minkowski weight $c\in M^*(\Sigma)$ the Minkowski weight representing the intersection cycle $[c]\cdot\Trop(Z)$, is commutative.
\end{thm}

\begin{proof}
Because all maps involved are graded, it is sufficient to show the commutativity for homogeneous elements. Let $d$ be the dimension of $Z$, then for $k>d$ the $k$-th components of the vertical maps are both zero and hence the diagram commutes in degrees greater $d$. Now assume that $c\in A^d(X)$. Let $t\in A^{n-d}(X)$ be the cocycle on $X$ corresponding to $[\overline Z]$ by Poincaré duality, that is the unique cocycle with $t\cap [X]=[\overline Z]$. Then its associated tropical cycle $[\mathcal I_T(t)]=\Trop(Z)$ is the tropicalization of $Z$ (see Section \ref{subsec:initial ideals}). Identifying both, $M^n(\Sigma)$ and $M^d(\widetilde\Sigma)$, with $\mathds Z$, we get
\begin{align*}
\mathcal I_{Z}(i^*c) 
= \deg(i^*c\cap[\overline Z]) 
&=\deg(c\cap [\overline Z])=\\
&=\deg((c\cup t)\cap [X])
=\mathcal I_T(c\cup t)
=\mathcal I_T(c)\cdot\Trop(Z),
\end{align*}
where the second equality uses the projection formula, and the last  equality uses that the ring structure of the Chow cohomology of complete toric varieties is compatible with tropical intersection products.

Now assume $c\in A^k(X)$ for some $k< d$ and let $\sigma\in\widetilde\Sigma^{(k)}$. By definition, the weight of $\mathcal I_Z(i^*c)$ at $\sigma$ is equal to $\deg(i^*c\cap[\overline Z\cap V(\sigma)])$. Denoting the inclusion $V(\sigma)\rightarrow X$ by $j$, this is equal to $\deg(j^*c\cap[\overline Z\cap V(\sigma)])$ by the projection formula. By Lemma \ref{mLem:Intersection and Closure Commute}, Proposition \ref{mProp:linear stratification}, and Corollary \ref{mCor:Tropiclization and is local}, we  know that $\overline Z\cap V(\sigma)$ is an admissible compactification of the linear variety $\overline Z\cap O(\sigma)$. So if we denote the inclusion $\overline Z\cap V(\sigma)\rightarrow V(\sigma)$ by $\kappa$, then the $k=d$ case applied to $j^*c$ yields
\begin{align*}
\deg(j^*c\cap[\overline Z\cap V(\sigma)])
&=
\deg(\kappa^*(j^*c)\cap[\overline Z\cap V(\sigma)])\\
&=
\mathcal I_{\overline Z\cap O(\sigma)}(\kappa^*(j^*c))\\
&=
\mathcal I_{O(\sigma)}(j^*c)\cdot\Trop(\overline Z\cap O(\sigma)),\\
\end{align*}
where the first equality again follows by the projection formula and in the last two expressions we identify $M^k(\Star_{\widetilde\Sigma}(\sigma))$ with $\mathds Z$. Using the projection formula one easily sees that $\mathcal I_{O(\sigma)}(j^*c)$ is equal to $\Star_{\mathcal I_T(c)}(\sigma)$, and by Corollary \ref{mCor:Tropiclization and is local} we know that $\Trop(\overline Z \cap O(\sigma))=\Star_{\Trop(Z)}(\sigma)$.  The locality of the intersection product \cite[Lemma 1.2]{Rau08} then implies that the intersection product of $\Star_{\mathcal I_T(c)}(\sigma)$ and $\Star_{\Trop(Z)}(\sigma)$ is equal to the weight of $\mathcal I_T(c)\cdot\Trop(Z)$ at $\sigma$, which is exactly what we wanted to show.
\end{proof}

Now we are able to prove our main result.

\begin{proof}[Proof of Theorem \ref{iThm:I_Z is ring isomorphism}]
We have already constructed the morphism $\mathcal I_Z$ and have seen that it is injective in Proposition \ref{mProp:Kronecker map is Isomorphism}. To see that it is also surjective, let $\Delta$ be a complete unimodular fan containing $\Sigma$ (which exists by \cite[Thm. 2.8, p. 75]{E96} and \cite[Thm. 11.1.9]{CLS}). Because $|\Sigma|=|\Trop(Z)|$, the closure of $Z$ in $X_\Delta$ is equal to that in $X_\Sigma$. Let $i:\overline Z\rightarrow X_\Delta$ denote the inclusion map and consider the diagram 
\begin{center}
\begin{tikzpicture}[auto]
\matrix[matrix of math nodes,row sep=1.1cm]{
|(AX)|A^*(X_\Delta) &[1.2cm] |(MS)| M^*(\Delta) \\
|(AZ)|A^*(\overline Z) & |(MwS)| M^*(\Sigma) \\
};
\begin{scope}[->,font=\footnotesize]
\draw (AX)--node{$\mathcal I_T$} (MS);
\draw (AZ)--node{$\mathcal I_{Z}$} (MwS);
\draw (AX)--node{$i^*$} (AZ);
\draw (MS)--node{$\_\cdot\Trop(Z)$} (MwS);
\end{scope}
\end{tikzpicture}
\end{center}
which is commutative by Theorem \ref{mThm:Pullback is Tropicalization}. Since $Z$ is linear, the fan $\Sigma$ has support equal to that of a matroid variety. We have seen in Theorem \ref{tThm:Intersection with Tropicalization is surjective} and Corollary \ref{tCor:Ring structure on Minkowski weights} that in this situation the intersection product on $Z^*(\Trop(Z))$ induces a ring structure on $M^*(\Sigma)$ and that the morphism on the right is a surjective ring homomorphism. We also know that the upper horizontal map $\mathcal I_T$ is an isomorphism by \cite[Thm 2.1]{FS97}. Together, these facts imply that $\mathcal I_Z$ is a ring isomorphism.
\end{proof}

\begin{cor}
\label{mCor:Chow groups of Z modulo torsion}
Let $k\in\mathds N$, and $Z$ and $\Sigma$ as in Theorem \ref{iThm:I_Z is ring isomorphism}. Then the $k$-th Chow group $A_k(\overline Z)_{\mathds Q}=A_k(\overline Z)\otimes_{\mathds Z}\mathds Q$ with rational coefficients is canonically isomorphic to the quotient of the $\mathds Q$-vector space with basis $\{e_\sigma\mid\sigma\in \Sigma^{(k)}\}$ by the subspace spanned by the elements
\begin{equation*}
\sum_{\mathclap{\sigma:\tau\leq\sigma\in\Sigma^{(k)}}}~~ \langle m,u_{\sigma,\tau}\rangle e_\sigma
\end{equation*}
for $\tau\in\Sigma^{(k+1)}$ and $m\in M(\sigma)$, where $u_{\sigma,\tau}$ denotes the lattice normal vector of $\sigma$ with respect to $\tau$.
\end{cor}

\begin{proof}
Denoting the vector space just described by $V=\Lin\{e_\sigma\mid \sigma\in\Sigma\}$, the dual $V^*$ of $V$ clearly is canonically isomorphic to $M^k(\Sigma)_{\mathds Q}$. This in turn is canonically isomorphic to $A^k(\overline Z)_{\mathds Q}$ by Theorem \ref{iThm:I_Z is ring isomorphism}. It follows from Proposition \ref{mProp:Kronecker map is Isomorphism} that this is isomorphic to $(A_k(\overline Z)_{\mathds Q})^*$. Dualizing the composite isomorphism $V^*\cong (A_k(\overline Z)_{\mathds Q})^*$ finishes the proof.
\end{proof}

As a final result we will prove that the intersection ring of $\Trop(Z)$ is isomorphic to the direct limit of the cohomology rings of admissible compactifications of $Z$.

\begin{cor}
\label{mCor:Tropical Intersection Theory is direct limit of algebraic intersection Theory}
Let $Z$ be a linear subvariety of the torus $T$ and let $\mathcal D$ be the directed set of unimodular fans in $\nr$ with support equal to $|\Trop(Z)|$. Then we have $\varinjlim A^*(\overline Z)\cong Z^*(\Trop(Z))$.
\end{cor}

\begin{proof}
As an immediate consequence of the fact that every fan has a unimodular refinement, we get the equality $\varinjlim M^*(\Sigma)\cong Z^*(\Trop(Z))$.  For all $\Delta \in\mathcal D$ we have $M^*(\Delta)\cong A^*(\overline Z)$ by
Theorem \ref{iThm:I_Z is ring isomorphism}, where $\overline Z$ is the closure of $Z$ in $X_\Delta$. Hence it is only left to show that whenever $\Sigma\in\mathcal D$ refines $\Delta\in\mathcal D$, and $\overline Z'$ denotes the closure of $Z$ in $X_\Sigma$, the diagram
\begin{center}
\begin{tikzpicture}[auto]
\matrix[matrix of math nodes,row sep=1.1cm]{
|(AZ)|A^*(\overline Z) &[1.2cm] |(MD)| M^*(\Delta) \\
|(AZ')|A^*(\overline Z') & |(MS)| M^*(\Sigma), \\
};
\begin{scope}[->,font=\footnotesize]
\draw (AZ)--node{$\mathcal I_Z$} (MD);
\draw (AZ')--node{$\mathcal I_{Z}$} (MS);
\draw (AZ)--node{$i^*$} (AZ');
\draw (MD)--(MS);
\end{scope}
\end{tikzpicture}
\end{center}
where $i\colon \overline Z'\to\overline Z$ is the morphism induced by the identity on $N$, and the vertical arrow
on the right is the refinement of Minkowski weights, is commutative. This boils down to showing that whenever $\sigma\in\Sigma$ and $\tau\in\Delta$ are two cones of the same codimension such that $\sigma\subseteq\tau$, then $i_*([\overline Z' \cap V (\sigma)]) = [Z \cap V (\tau)]$. In this situation, we have $N_\sigma = N_\tau$ and
the restriction of $i$ to $\overline Z'\cap V (\sigma)$ is induced by the toric morphism $V (\sigma)\to V(\tau)$ coming
from the identity on $N(\sigma) = N(\tau)$. As $\overline Z \cap V (\sigma)$ is the closure of $Z \cap O(\sigma)$ by Lemma \ref{mLem:Intersection and Closure Commute}, it is sufficient to show that the push-forward of $[\overline Z'\cap O(\sigma)]$ under the induced map $j\colon O(\sigma)\to O(\tau)$ is equal to $[\overline Z \cap O(\tau)]$. But this is clear because $j$ is an isomorphism, $j(\overline Z' \cap O(\sigma)) \subseteq \overline Z \cap O(\tau)$, and $Z \cap O(\sigma)$ and $\overline Z' \cap O(\tau)$ are both irreducible of the same dimension.
\end{proof}

\bibliography{}

Andreas Gross, Fachbereich Mathematik, Technische Universität Kaiserslautern, Postfach 3049, 67653 Kaiserslautern, Germany \texttt{agross@mathematik.uni-kl.de}

\end{document}